\documentclass[12pt]{article}

\usepackage{amsmath,amsthm,amssymb,mathdots,enumerate,hyperref} 

\usepackage{xcolor}

\usepackage{faktor}
\usepackage{array}
\usepackage{tikz}
\usepackage{tikz-cd}

\usepackage{setspace}
\allowdisplaybreaks

\newtheorem{theorem}{Theorem}[section]
\newtheorem{proposition}[theorem]{Proposition}

\newtheorem{lemma}[theorem]{Lemma}
\newtheorem{corollary}[theorem]{Corollary}

\theoremstyle{remark}
\newtheorem{remark}[theorem]{Remark}

\numberwithin{equation}{section}

 \usetikzlibrary{decorations.pathmorphing} 

\DeclareFontEncoding{OT2}{}{}
\DeclareFontSubstitution{OT2}{cmr}{m}{n}
\DeclareFontFamily{OT2}{cmr}{\hyphenchar\font45 }
\DeclareFontShape{OT2}{cmr}{m}{n}{<5><6><7><8><9>gen*wncyr<10><10.95><12><14.4><17.28><20.74><24.88>wncyr10}{}
\DeclareFontShape{OT2}{cmr}{b}{n}{<5><6><7><8><9>gen*wncyb<10><10.95><12><14.4><17.28><20.74><24.88>wncyb10}{}
\DeclareMathAlphabet{\mathcyr}{OT2}{cmr}{m}{n}
\DeclareMathAlphabet{\mathcyb}{OT2}{cmr}{b}{n}
\SetMathAlphabet{\mathcyr}{bold}{OT2}{cmr}{b}{n}

\newcommand{\C}{{\mathbb C}}
\newcommand{\Q}{{\mathbb Q}}

\DeclareMathOperator{\wt}{wt}
\DeclareMathOperator{\dep}{dep}
\DeclareMathOperator{\height}{ht}
\newcommand{\kk}{{\bf k}}

\newcommand{\z}{\xi}
\newcommand{\zn}{z}
\newcommand{\bz}{\overline{z}}

\newcommand{\es}{ \operatorname{e}}

\title{Special values of finite multiple harmonic $q$-series \\at roots of unity}
\author{
Henrik Bachmann\footnote{email : henrik.bachmann@math.nagoya-u.ac.jp, Nagoya University},\,  
Yoshihiro Takeyama\footnote{email :  takeyama@math.tsukuba.ac.jp, University of Tsukuba},\,  
Koji Tasaka\footnote{email : tasaka@ist.aichi-pu.ac.jp, Aichi Prefectural University}
}
\date{}

\begin{document}

\maketitle

\begin{abstract}
We study special values of finite multiple harmonic $q$-series at roots of unity. These objects were recently introduced by the authors and it was shown that they have connections to finite and symmetric multiple zeta values and the Kaneko-Zagier conjecture. In this note we give new explicit evaluations for finite multiple harmonic $q$-series at roots of unity and prove Ohno-Zagier-type relations for them. 
\end{abstract}


\section{Introduction}

Finite multiple harmonic $q$-series at roots of unity were introduced by the authors in \cite{BTT}. The motivation to study these object is their connection to finite and symmetric multiple zeta values and a reinterpretation of the conjecture of Kaneko and Zagier in \cite{KanekoZagier}. In this paper we concentrate on the finite multiple harmonic $q$-series at roots of unity themselves and give explicit evaluations and Ohno-Zagier-type relations for them. 

For an index $\kk=(k_1,\dots,k_r)\in (\mathbb{Z}_{\ge 1})^r$, a natural number $n\geq 1$ and a complex number  $q$ satisfying $q^{m}\neq 1$ for $n>m>0$, the \emph{finite multiple harmonic $q$-series} and their star-versions are defined by
\begin{align*}\begin{split} \label{eq:defzn}
\zn_n(\kk;q) = \zn_n(k_1,\dots,k_r;q) =\sum_{n > m_1 > \dots > m_r >0} 
\frac{q^{(k_1-1) m_1} \dots q^{(k_r-1) m_r}}{ [m_1]_q^{k_1} \dots [m_r]_q^{k_r}}\,,\\
\zn^\star_n(\kk;q) = \zn^\star_n(k_1,\dots,k_r;q) =\sum_{n > m_1 \geq \dots \geq m_r >0} 
\frac{q^{(k_1-1) m_1} \dots q^{(k_r-1) m_r}}{ [m_1]_q^{k_1} \dots [m_r]_q^{k_r}}\,,
\end{split}
\end{align*}
where $[m]_q$ is the $q$-integer $[m]_q= \frac{1-q^m}{1-q}$. 
For a primitive $n$-th root of unity $\zeta_{n}$ the values $\zn_n(\kk;\zeta_{n})$ and $\zn^\star_n(\kk;\zeta_{n})$ are elements in the cyclotomic field $\Q(\zeta_n)$. 
The first result of this paper are the following evaluations of $z_n(k,\dots,k;\zeta_n)$.
\begin{theorem} \label{thm:zkkk}
For all $k,r\geq 1$ and any $n$-th primitive root of unity $\zeta_n$ we have $z_n({\{k\}^r}; \zeta_n) \in (1-\zeta_n)^{kr} \Q$ and in particular
\begin{align*}
z_n({\{1\}^r}; \zeta_n) &= \frac{1}{n} \binom{n}{r+1} (1-\zeta_n)^{r}\,,  \\ 
z_n({\{2\}^r}; \zeta_n) &= \frac{(-1)^r}{n \cdot (r+1)} \binom{n+r}{2r+1}(1-\zeta_n)^{2r}\,, \\
z_n({\{3\}^r}; \zeta_n) &= \frac{1}{n^2(r+1)} \left( \binom{n+2r+1}{3r+2} +(-1)^r \binom{n+r}{3r+2}\right)(1-\zeta_n)^{3r} \,.
\end{align*}
\end{theorem}

The second result of this paper are  Ohno-Zagier-type relations for $z_n(\kk; \zeta_n)$ and $z^\star_n(\kk; \zeta_n)$. For an index $\kk=(k_1,\dots,k_r)$ denote by $\wt(\kk) = k_1+\dots+k_r$ the \emph{weight} and by $\dep(\kk)=r$ the \emph{depth}.
In \cite{OhnoZagier}, Ohno and Zagier introduced the \emph{height} by 
$\height(\kk)=\#\left\{ a \in \{1, 2, \ldots , r\} \, | \, k_{a}\ge 2\right\}$. 
They proved an explicit formula for the generating function of the sum 
of multiple zeta values of fixed weight, depth and height. 

To state the Ohno-Zagier-type relations for finite multiple harmonic $q$-series at roots of unity we first define their modified versions by
\begin{equation*}
\bz_{n}(\mathbf{k};\zeta_n) := (1-\zeta_n)^{-\wt(\kk)} z_{n}(\mathbf{k};\zeta_n)\,,\qquad \bz^{\star}_{n}(\mathbf{k};\zeta_n) := (1-\zeta_n)^{-\wt(\kk)} z^{\star}_{n}(\mathbf{k};\zeta_n)\,.
\end{equation*}
For positive integers $k,r,s\ge0$, let $I(k, r, s)$ be the set of indices of weight $k$, depth $r$ and height $s$. We denote the generating function of the sum of the modified versions of fixed weight, depth and height  by
\begin{equation}\label{eq:ohno-zagier-generating-funtcion}
\begin{aligned}
& 
F_{n}(x, y, z)=1+\sum_{r=1}^{\infty}\sum_{s=0}^{r}\sum_{k=r+s}^{\infty}
\left(\sum_{\mathbf{k} \in I(k, r, s)}\bz_{n}(\mathbf{k};\zeta_n)\right)x^{k-r-s}y^{r-s}z^{s}, \\ 
& 
F_{n}^{\star}(x, y, z)=1+\sum_{r=1}^{\infty}\sum_{s=0}^{r}\sum_{k=r+s}^{\infty}
\left(\sum_{\mathbf{k} \in I(k, r, s)}\bz_{n}^{\star}(\mathbf{k};\zeta_n)\right)x^{k-r-s}y^{r-s}z^{s}. 
\end{aligned}
\end{equation} 

\begin{theorem}\label{thm:ohno-zagier}
The sum of $\bz_{n}(\kk)$ (or $\bz_{n}^{\star}(\kk)$) 
of fixed weight, depth and height is a rational number. 
More explicitly we have  
\begin{align*}
F_{n}(x, y, z)=U_{n}(x, y, z), \qquad 
F_{n}^{\star}(x, y, z)=U_{n}(x, -y, -z)^{-1},   
\end{align*}
where 
\begin{align}
\begin{aligned}
& 
U_{n}(x, y, z)=\frac{x}{(1+x)^{n}-1} \\ 
& {}\times 
\sum_{\substack{a, b \ge 0 \\ a+b\le n-1}}\frac{1}{n-a-b}
\binom{n-a-1}{b}\binom{n-b-1}{a}
(1+x)^{a}(1+y)^{b}(xy-z)^{n-1-a-b}.  
\end{aligned}
\label{eq:ohno-zagier-def-U}
\end{align}
\end{theorem}
In \cite[Theorem 1.2]{BTT} it was shown that for any index $\kk \in (\mathbb{Z}_{\ge 1})^r$, the limit
\begin{align*}
\z(\kk)=\lim_{n\to \infty}\zn_n(\kk;e^{2\pi i/n})
\end{align*}
exists in $\C$. The real part of $\z(\kk)$ has a connection to symmetric multiple zeta values. As an application of Theorems \ref{thm:zkkk} and \ref{thm:ohno-zagier} we will give explicit evaluations of $\z(k,\dots,k)$ for $k=1,2,3$ and prove Ohno-Zagier-type relations for these complex numbers (see Proposition \ref{prop:ohno-zagier-z}). As a consequence we also obtain the following sum formula for $\z(\kk)$.

\begin{corollary}\label{cor:xisum}
For all $k\geq r \geq 1$ we have
\begin{align*}
\sum_{\kk \in I(k, r)}\z(\kk)=-\frac{(-2\pi i)^k}{(k+1)!}\sum_{j=1}^r \binom{k+1}{j} B_{k+1-j}   \,,  
\end{align*} 
where $I(k, r)$ denotes the set of indices of weight $k$ and depth $r$ and  where $B_k$ is the $k$-th Bernoulli number with the convention $B_1 = -\frac{1}{2}$.
\end{corollary}

The contents of this paper are as follows. In Section \ref{sec:zkkk} we consider the values $z_n(k,\dots,k;\zeta_n)$ and give the proof of Theorem \ref{thm:zkkk}. The Ohno-Zagier-type relations and the proof of Theorem \ref{thm:ohno-zagier} will be given in Section \ref{sec:ohnozagier}. In Section \ref{sec:application} we present evaluations and the Ohno-Zagier-type relations for the values $\xi(\kk)$ and give the proof of Corollary \ref{cor:xisum}.\\

\noindent\textbf{Acknowledgments} 
This work was partially supported by JSPS KAKENHI Grant Numbers 16F16021, 16H07115, 18K13393 and 18K03233.


\section{The values $z_n(k,\dots,k;\zeta_n)$}\label{sec:zkkk}
 For depth one it was shown in \cite{BTT} (see also Corollary \ref{cor:generating-function-depth-one-and-sumformula}), that for all $k\geq 1$ we have
\begin{equation}\label{eq:degbern}
\zn_n(k;\zeta_n) = -\frac{\beta_k(n^{-1}) n^k}{k!}(1-\zeta_n)^{k} \,,
\end{equation}
where $\beta_k(x) \in \Q[x]$ is the \emph{degenerate Bernoulli number} defined by Carlitz in \cite{C}. Since $\beta_k(n^{-1}) n^k \in \Q[n]$, the $\bz_n(k;\zeta_n)$ are polynomials in $n$. For example,  we have
\begin{equation*}
\begin{aligned}
&\zn_n(1;\zeta_n) =\frac{n-1}{2} (1-\zeta_n)\,,\quad \zn_n(2;\zeta_n)=-\frac{n^2-1}{12} (1-\zeta_n)^2\,,\\
&\zn_n(3;\zeta_n) =\frac{n^2-1}{24} (1-\zeta_n)^3 \,, \quad
\zn_n(4;\zeta_n)=\frac{(n^2-1)(n^2-19)}{720} (1-\zeta_n)^4. 
\end{aligned}
\end{equation*}
The formula \eqref{eq:degbern} can be seen as an analogue of the formula by Euler for the Riemann zeta values $\zeta(k) =-\frac{B_{k}}{2k!} (-2\pi i)^{k}$ for even $k$. The multiple zeta values defined for $k_1\geq 2, k_2,\dots,k_r\geq 1$ by
\begin{align*}\label{eq:mzv}
\zeta(k_1,\ldots,k_r)=\sum_{m_1>\cdots>m_r>0 }\frac{1}{m_1^{k_1}\cdots m_r^{k_r}}.
\end{align*} 
also have explicit evaluations for even $k$ and $k_1=\dots=k_r=k$. Theorem \ref{thm:zkkk} can be seen as an analogue of these formulas.  We have for example
\[  \zeta(\{2\}^r) = \frac{\pi^{2r}}{(2r+1)!}\,, \]
which is an easy consequence of the product formula of the sine
\[ \sum_{j=0}^\infty \zeta( \{ 2 \}^j) T^{2j+1} = T \prod_{m=1}^\infty \left( 1 + \frac{T^2}{m^2} \right) = \frac{\sin\left( \pi i T\right)}{\pi i } = \sum_{j=0}^{\infty}  \frac{\pi^{2j}}{(2j+1)!}  T^{2j+1} \,. \]
We will use a similar idea for the proof of Theorem \ref{thm:zkkk}.

\begin{proof}[Proof of Theorem \ref{thm:zkkk}]

We will prove explicit formulas for the modified versions $\bz({\{k\}^r};\zeta_n)$ and then obtain the result in Theorem \ref{thm:zkkk} by multiplying with $(1-\zeta_n)^{kr}$.
By definition of the values $\bz_n({\{k\}^r};\zeta_n)$ we derive
\begin{align*}
\sum_{r=0}^{n-1} \bz_n({\{k\}^r};\zeta_n) \,X^r =  \prod_{j=1}^{n-1} \left(1 + \frac{\zeta_n^{(k-1)j}}{(1-\zeta_n^j)^k} X\right) = \frac{1}{n^{k}} \prod_{j=1}^{n-1} \big( (1-\zeta_n^j)^k + \zeta_n^{(k-1)j} X \big) \,,
\end{align*}
where we used $\prod_{j=1}^{n-1}(1-\zeta_n^j)=n$ at the last equation. For $i=1,\dots,k$ we denote by $\alpha_i$ the roots of the polynomial $(1-Y)^k + Y^{(k-1)}X$, i.e. 
\[\prod_{i=1}^k \left( \alpha_i - Y \right) :=  (1-Y)^k + Y^{(k-1)}X\,.\]
The $\alpha_i$  can not be computed explicitly but by definition they satisfy
\begin{equation}\label{eq:elsyminalpha}
\es_j(\alpha_1,\dots,\alpha_k) = \binom{k}{j}  +  (-1)^{k-1} \delta_{j,1}X \,,
\end{equation}
with the elementary symmetric polynomials $\es_j(x_1,\dots,x_k)$ defined by 
 \[\sum_{j=0}^k (-1)^j \es_j(x_1,\dots,x_k) \,T^{k-j} := \prod_{j=1}^k (T-x_j) \,. \]
Using $\prod_{j=1}^{n-1} (\alpha_i - \zeta_n^j) = \frac{\alpha_i^n-1}{\alpha_i-1}$ and $\prod_{i=1}^k \left( \alpha_i - 1\right)=X$ we obtain 
\begin{align*}
\frac{1}{n^k}\prod_{j=1}^{n-1} \prod_{i=1}^k (\alpha_i-\zeta_n^j) = \frac{1}{n^k} \prod_{i=1}^k \frac{\alpha_i^n-1}{\alpha_i-1} = \frac{1}{n^k X}  \prod_{i=1}^k (\alpha_i^n-1) \,.
\end{align*}
Now consider also the generating series of this over all $n$
\begin{align*}
\sum_{n=1}^\infty& n^{k-1} Y^n \sum_{r=0}^{n-1} \bz({\{k\}^r};\zeta_n) \,X^r = \frac{1}{X} \sum_{n=1}^\infty \frac{Y^n}{n} \prod_{i=1}^k (\alpha_i^n-1) \\ 
&= \frac{1}{X} \sum_{n=1}^\infty \frac{Y^n}{n} \left((-1)^k + \sum_{l=1}^k (-1)^{k-l} \sum_{1\leq i_1 < \dots < i_l \leq k} (a_{i_1} \dots a_{i_l})^n\right) \\
&= \frac{(-1)^{k-1}}{X} \left(-\sum_{n=1}^\infty \frac{Y^n}{n}+\sum_{l=1}^k (-1)^{l-1} \sum_{1\leq i_1 < \dots < i_l \leq k} \frac{(a_{i_1} \dots a_{i_l} Y)^n}{n} \right) \\
&= \frac{(-1)^{k-1}}{X} \left(\log(1-Y) + \sum_{l=1}^k (-1)^{l}  \sum_{1\leq i_1 < \dots < i_l \leq k} \log(1-a_{i_1} \dots a_{i_l} Y) \right) \,. 
\end{align*}
Defining for $0 \leq l \leq k$  the polynomials $F_{k,l}(X,Y)$  by
\begin{align*}
F_{k,0}(X,Y) &= 1-Y,\\
F_{k,l}(X,Y) &= \prod_{1\leq i_1 < \dots < i_l \leq k} (1-a_{i_1} \dots a_{i_l} Y)\qquad (1 \leq l \leq k)\,,
\end{align*}
we obtain
\begin{equation}\label{eq:genfctkkk}
\sum_{n=1}^\infty n^{k-1} Y^n \sum_{r=0}^{n-1} \bz({\{k\}^r};\zeta_n) \,X^r = \frac{(-1)^{k-1}}{X} \log\left(\prod_{l=0}^k F_{k,l}(X,Y)^{(-1)^l}  \right) \,.
\end{equation}
The coefficients of $Y^m$ in $F_{k,l}(X,Y)$ are symmetric polynomials in the $\alpha_1\,\dots,\alpha_k$ and can therefore be written in terms of the elementary symmetric polynomials $\es_j(\alpha_1,\dots,\alpha_k)$ in \eqref{eq:elsyminalpha}. In particular this shows that $\bz({\{k\}^r};\zeta_n) \in \Q$ for all $k,r\geq1$. 
For the cases $k=1,2,3$ we obtain
\begin{align*}
F_{1,0}(X,Y) &= F_{1,1}(X,Y)= 1-Y\,,\\
F_{2,0}(X,Y)&=F_{2,2}(X,Y)= 1-Y\,,\quad F_{2,1}(X,Y) = (1-Y)^2 + XY\,,\\
F_{3,0}(X,Y)&=F_{3,3}(X,Y)= 1-Y\,,\quad F_{3,1}(X,Y) = (1-Y)^3 - XY\,,\\
F_{3,2}(X,Y) &= (1-Y)^3 + XY^2\,.
\end{align*}
Plugging these into \eqref{eq:genfctkkk}, we get for example for the $k=3$ case
\begin{align*}
 &\frac{1}{X} \log\left(\prod_{l=0}^3 F_{3,l}(X,Y)^{(-1)^l}  \right) =  \frac{1}{X} \log\left( \frac{(1-Y)^3 + XY^2}{(1-Y)^3-XY} \right) \\
 &= \frac{1}{X}\left(  \log\left(1+ \frac{XY^2}{(1-Y)^3}\right) - \log\left(1- \frac{XY}{(1-Y)^3}\right)\right) \\
 &= \frac{1}{X} \left( - \sum_{m=1}^\infty \frac{(-1)^{m-1}}{m} \left(\frac{XY^2}{(1-Y)^3}\right)^m + \sum_{m=1}^\infty \frac{1}{n} \left(\frac{XY}{(1-Y)^3} \right)^m \right)
\end{align*}
Using for $l\geq 1$ the formula
\[\frac{1}{(1-Y)^l} = \sum_{j=0}^\infty \binom{l-1+j}{l-1} Y^j\,,\]
we obtain by \eqref{eq:genfctkkk} the formula for $z_n(\{3\}^r;\zeta_n)$ as stated in Theorem \ref{thm:zkkk}. The formulas for $z_n(\{1\}^r;\zeta_n)$ and $z_n(\{2\}^r;\zeta_n)$ follow similarly.
\end{proof}
Notice that above proof works for every $k$, i.e. it is possible to obtain explicit formulas for $z_n(\{k\}^r;\zeta_n)$. But these formulas might get quite complicated for larger $k$.
  
\begin{remark}
Additionally to the above formulas we have the following observations. For all $a,b \geq 0$ it seems to be
\begin{align*}
z_n(\{1\}^a,2,\{1\}^b;\zeta_n) +z_n(\{1\}^b,2,\{1\}^a;\zeta_n) &\overset{?}{=} -\frac{1}{n} \binom{n+1}{a+b+3} (1-\zeta_n)^{a+b+2} \,, \\
z_n(\{2\}^a,3,\{2\}^b;\zeta_n) +z_n(\{2\}^b,3,\{2\}^a;\zeta_n) &\overset{?}{=} -\frac{(-1)^{a+b}}{(a+b+2)n} \binom{n+a+b+1}{2(a+b)+3} (1-\zeta_n)^{2(a+b)+3} \,,
\end{align*}  
but so far the authors were not able to prove these formulas. 
\end{remark}


\section{Ohno-Zagier-type relation}\label{sec:ohnozagier}

\subsection{Proof of Ohno-Zagier-type relation}\label{subsec:Ohno-Zagier-proof}

The proof of Theorem \ref{thm:ohno-zagier} is essentially the same as that of 
Ohno-Zagier's relation for multiple zeta values \cite{OhnoZagier} (see also \cite{OkudaTakeyama, Takeyama}) except for the following point. 
In both proofs the main ingredient is the generating function of type \eqref{eq:ohno-zagier-generating-funtcion}. 
While Ohno-Zagier's relation follows from its explicit formula in terms of the hypergeometric function and its special value, our formula follows from a consistency condition for the $q$-difference equation of which the generating function gives a unique polynomial solution. 

\subsubsection{A $q$-analogue of finite multiple polylogarithm of one variable}

Throughout this subsection, $q$ denotes a complex number.
For a positive integer $n$ and an index $\kk=(k_{1}, \ldots , k_{r})$, 
we define the function $L_{\kk}^{(n)}(t; q)$ and its star version $L_{\kk}^{\star, (n)}(t; q)$ by 
\begin{align*}
L_{\kk}^{(n)}(t;q)&=\sum_{n>m_{1}>\cdots >m_{r}>0}\frac{t^{m_{1}}}{\prod_{i=1}^{r}(1-q^{m_{i}})^{k_i}}, \\ 
L_{\kk}^{\star, (n)}(t;q)&=\sum_{n>m_{1}\ge \cdots \ge m_{r}>0}\frac{t^{m_{1}}}{\prod_{i=1}^{r}(1-q^{m_{i}})^{k_i}}. 
\end{align*}
Note that when $q\in \C $ with $|q|<1$ or $q=\zeta_n$ a $n$-th primitive root of unity, 
$L_{\kk}^{(n)}(t; q)$ and $L_{\kk}^{\star, (n)}(t; q)$ are polynomials in $t$ whose degree is less than $n$, 
and they satisfy $L_{\mathbf{k}}^{(n)}(t; q)=O(t^{\mathrm{dep}(\mathbf{k})}) $ and 
$L_{\mathbf{k}}^{\star, (n)}(t; q)=O(t)$ as $t \to 0$. 
Taking the limit as $n \to \infty$ we recover the $q$-multiple polylogarithms defined in \cite{Bradley1, Zhao}. 

Hereafter we fix a positive integer $n$ and omit the superscript $(n)$ in 
$L_{\mathbf{k}}^{(n)}(t; q)$ and $L_{\mathbf{k}}^{\star, (n)}(t; q)$. 

We denote by $I_{0}(k, r, s)$ the set of admissible indices 
of weight $k$, depth $r$ and height $s$. 
For non-negative integers $k, r$ and $s$, set 
\begin{align*}
& 
G(k, r, s; t;q)=\sum_{\mathbf{k}\in I(k, r, s)}L_{\mathbf{k}}(t; q), \qquad 
G_{0}(k, r, s; t;q)=\sum_{\mathbf{k}\in I_{0}(k, r, s)}L_{\mathbf{k}}(t; q), 
\\ 
& 
G^\star(k, r, s; t;q)=\sum_{\mathbf{k}\in I(k, r, s)}L^{\star}_{\mathbf{k}}(t; q), \qquad 
G^{\star}_{0}(k, r, s; t;q)=\sum_{\mathbf{k}\in I_{0}(k, r, s)}L^{\star}_{\mathbf{k}}(t; q), 
\end{align*}
where $I(k, r, s)$ is the set of indices of weight $k$, depth $r$ and height $s$. 
If $I(k, r, s)$ (resp. $I_{0}(k, r, s)$) is empty, 
$G(k, r, s; t;q)$ and $G^{\star}(k, r, s; t;q)$ 
(resp. $G_{0}(k, r, s; t;q)$ and $G^\star(k, r, s; t;q)$) are zero, while 
we set by definition $G(0,0,0; t;q)=G^{\star}(0,0,0; t;q)=1$ and $G_{0}(0,0,0; t;q)=G^{\star}_0(0,0,0; t;q)=0$. 

Consider the generating functions 
\begin{align*}
& 
\Phi(t;q)=\Phi(u, v, w; t;q)=\sum_{k, r, s\ge 0}G(k, r, s; t;q)u^{k-r-s}v^{r-s}w^{s}, \\ 
& 
\Phi_{0}(t;q)=\Phi_{0}(u, v, w; t;q)=\sum_{k, r, s\ge 0}G_{0}(k, r, s; t;q)u^{k-r-s}v^{r-s}w^{s-1} ,\\
& 
\Phi^{\star}(t;q)=\Phi^{\star}(u, v, w; t;q)=\sum_{k, r, s\ge 0}G^{\star}(k, r, s; t;q)u^{k-r-s}v^{r-s}w^{s}, \\ 
& 
\Phi_{0}^{\star}(t;q)=\Phi_{0}^{\star}(u, v, w; t;q)=\sum_{k, r, s\ge 0}G_{0}^{\star}(k, r, s; t;q)u^{k-r-s}v^{r-s}w^{s-1}. 
\end{align*}
Note that they are polynomials in $t$ whose coefficients are formal power series in $u, v$ and $w$, 
and whose degree is less than $n$. 
Moreover, from the definition, it holds that 
$\Phi_{0}(0;q)=\Phi_{0}^{\star}(0;q)=0$.

The coefficients of the generating functions $\Phi(1;q)$ and $\Phi^{\star}(1;q)$ can be written 
in terms of $\bz_n(\kk;q)$ and $\bz_n^{\star}(\kk;q)$.

\begin{lemma}\label{lem:ohno-zagier-L-to-z}
Set
\begin{align}
u=\frac{x}{1+x}, \quad 
v=y-\frac{z}{1+x}, \quad 
w=\frac{z}{(1+x)^{2}}.   
\label{eq:ohno-zagier-proof-transform}
\end{align} 
Then we have
\begin{align}
\Phi(u, v, w; 1;q)=1+\sum_{r=1}^{\infty}\sum_{s=0}^{r}\sum_{k=r+s}^{\infty}
\left(\sum_{\kk \in I(k, r, s)}\bz_{n}(\kk;q)\right)x^{k-r-s}y^{r-s}z^{s}  
\label{eq:ohno-zagier-gen} 
\end{align}
and 
\begin{align}
\Phi^\star(u, v, w; 1;q)=1+\sum_{r=1}^{\infty}\sum_{s=0}^{r}\sum_{k=r+s}^{\infty}
\left(\sum_{\kk \in I(k, r, s)}\bz_{n}^{\star}(\kk;q)\right)x^{k-r-s}y^{r-s}z^{s}.  
\label{eq:ohno-zagier-gens} 
\end{align} 
\end{lemma}

\begin{proof}
Since we have
\begin{align*}
\frac{1}{(1-q)^{k}}=\sum_{a=1}^{k}\binom{k-1}{a-1}\frac{q^{(a-1)m}}{(1-q^{m})^{a}} \quad 
(k \ge 1, \, m\ge 1),  
\end{align*}
it holds that 
\begin{align}
L_{k_{1}, \ldots , k_{r}}(1;q)=\sum_{a_{1}=1}^{k_{1}}\cdots \sum_{a_{r}=1}^{k_{r}}
\left\{\prod_{j=1}^{r}\binom{k_{j}-1}{a_{j}-1}\right\}
\bz_{n}(a_{1}, \ldots , a_{r};q). 
\label{eq:ohno-zagier-proof-transform-proof}
\end{align} 
Using the above formula we find that 
\begin{align*}
& 
\Phi(u, v, w; 1;q) \\  
&=\sum_{r=1}^{n-1}\left(\frac{v}{u}\right)^{r}  
\sum_{a_{1}, \ldots , a_{r}\ge 1} \bz_{n}(a_{1}, \ldots , a_{r};q)\prod_{j=1}^{r}\left( 
\sum_{k=a_{j}}^{\infty}\binom{k-1}{a_{j}-1}u^{k}\left(\frac{w}{uv}\right)^{\theta(k\ge 2)}\right), 
\end{align*}
where $\theta(\mathrm{P})=1$ if P is true and 
$\theta(\mathrm{P})=0$ if P is false. 
It holds that 
\begin{align*}
& 
\sum_{k=a}^{\infty}\binom{k-1}{a-1}u^{k}\left(\frac{w}{uv}\right)^{\theta(k\ge 2)} \\ 
&=\left(\frac{u}{1-u}\right)^{a}\left(1-u+\frac{w}{v}\right)
\left(\frac{w}{u(v+w-uv)}\right)^{\theta(a\ge 2)} \qquad (a \ge 1). 
\end{align*}
Thus we get 
\begin{align*}
& 
\Phi(u, v, w; 1;q) \\ 
&=\sum_{\substack{\mathbf{a} \\ \dep(\mathbf{a})\le n-1}} 
\bz_{n}(\mathbf{a};q)
\left(\frac{u}{1-u}\right)^{\wt(\mathbf{a})}
\left\{\frac{v}{u}\left(1-u+\frac{w}{v}\right)\right\}^{\dep(\mathbf{a})}
\left(\frac{w}{u(v+w-uv)}\right)^{\height(\mathbf{a})}.  
\end{align*}
Changing the variables $u, v$ and $w$ to $x, y$ and $z$ by \eqref{eq:ohno-zagier-proof-transform}, 
we get the desired formula \eqref{eq:ohno-zagier-gen}.

For the star version, the formula \eqref{eq:ohno-zagier-proof-transform-proof} still holds 
if $L_{k_{1}, \ldots , k_{r}}(1;q)$ and $\bz_{n}(a_{1}, \ldots , a_{r};q)$ 
are replaced by 
$L_{k_{1}, \ldots , k_{r}}^{\star}(1;q)$ and $\bz^{\star}_{n}(a_{1}, \ldots , a_{r};q)$, respectively.
Then one can check the equality \eqref{eq:ohno-zagier-gens} in the same way.
\end{proof}

We define the $q$-difference operator $\mathcal{D}_{q}$ by 
\begin{align*}
(\mathcal{D}_{q}f)(t)=\frac{1}{t}\left(f(t)-f(qt)\right).  
\end{align*}
The above generating functions satisfy the following $q$-difference equations.

\begin{lemma}
We have
\begin{align}
& 
\begin{aligned}
& 
qt(1-t)\mathcal{D}_{q}^{2}\Phi_{0}(t;q)+\left\{(1-q-u)(1-t)-vt\right\}\mathcal{D}_{q}\Phi_{0}(t;q)+
(uv-w)\Phi_{0}(t;q) 
\\
&=1-\Phi(1;q)t^{n-1},
\label{eq:ohno-zagier-proof1} 
\end{aligned}
\\ 
& 
\begin{aligned}
& 
qt^{2}(1-t)\mathcal{D}^{2}_{q}\Phi_{0}^{\star}(t;q)+t\left\{(1-q-u)(1-t)-v\right\}\mathcal{D}_{q}\Phi_{0}^{\star}(t;q)+
(uv-w)\Phi_{0}^{\star}(t;q)
\\ 
&=t-\Phi^{\star}(1;q)t^{n}.
\nonumber
\end{aligned}
\end{align}
\end{lemma}

\begin{proof}
We make use of the following formulas:  
\begin{align*}
\mathcal{D}_{q}L_{k_{1}, \ldots , k_{r}}(t)=\left\{ 
\begin{array}{ll}
\displaystyle 
\frac{1}{t}L_{k_{1}-1, k_{2}, \ldots , k_{r}}(t) & (k_{1}\ge 2) \\
\displaystyle 
\frac{1}{1-t}\left(L_{k_{2}, \ldots , k_{r}}(t)-t^{n-1}L_{k_{2}, \ldots  k_{r}}(1)\right) &  (k_{1}=1) 
\end{array}
\right. 
\end{align*}
and
\begin{align*}
\mathcal{D}_{q}L_{k_{1}, \ldots , k_{r}}^{\star}(t)=\left\{ 
\begin{array}{ll}
\displaystyle 
\frac{1}{t}L_{k_{1}-1, k_{2}, \ldots , k_{r}}^{\star}(t) & (k_{1}\ge 2) \\
\displaystyle 
\frac{1}{t(1-t)}\left(L_{k_{2}, \ldots , k_{r}}^{\star}(t)-t^{n-1}L_{k_{2}, \ldots , k_{r}}^{\star}(1)\right) & 
(k_{1}=1, \, r\ge 2) \\ 
\displaystyle 
\frac{1-t^{n-1}}{1-t} & (k_{1}=1, \, r=1). 
\end{array}
\right. 
\end{align*}
Since the proof is similar, we only prove \eqref{eq:ohno-zagier-proof1}.
From the above formulas, we see that
\begin{align*}
& 
\mathcal{D}_{q}G_{0}(k, r, s; t;q) \\ 
&=\frac{1}{t}\left( 
G_{0}(k-1, r, s; t;q)+G(k-1, r, s-1; t;q)-G_{0}(k-1, r, s-1; t;q)
\right), \\ 
& 
\mathcal{D}_{q}\left(G(k, r, s; t;q)-G_{0}(k, r, s; t;q)\right) \\ 
&=\frac{1}{1-t}\left(
G(k-1,r-1,s; t;q)-t^{n-1}G(k-1,r-1,s; 1;q) 
\right),
\end{align*}
and hence 
\begin{align*}
& 
\mathcal{D}_{q}\Phi_{0}(t;q)=
\frac{1}{t}\left\{(u-\frac{w}{v})\Phi_{0}(t;q)+\frac{1}{v}\left(\Phi(t;q)-1\right)\right\},  
\\ 
& 
\mathcal{D}_{q}\Phi(t;q)-w\mathcal{D}_{q}\Phi_{0}(t;q)=\frac{v}{1-t}\left(\Phi(t;q)-t^{n-1}\Phi(1;q)\right).  
\end{align*}
Eliminate $\Phi(t;q)$ using the Leibniz rule 
\begin{align*}
\mathcal{D}_{q}(f(t)g(t))=\mathcal{D}_{q}(f(t))g(t)+f(qt)\mathcal{D}_{q}(g(t)).  
\end{align*}
Then we get \eqref{eq:ohno-zagier-proof1}.
\end{proof}

\begin{proposition}\label{prop:ohno-zagier-Phi}
We have
\begin{align}
\Phi(1;q)=\prod_{j=1}^{n-1}\frac{P(q^{j})}{(1-q^{j})(1-u-q^{j})}, 
\label{eq:ohno-zagier-proof4}
\end{align}
and
\begin{align}
\Phi^{\star}(1;q)=\prod_{j=1}^{n-1}\frac{(1-q^{j})(1-u-q^{j})}{P^{\star}(q^{j})},  
\label{eq:ohno-zagier-proof5}
\end{align}
where $P(X)$ and $P^{\star}(X)$ are the quadratic polynomials given by  
\begin{align}
P(X)=(1-u-X)(1+v-X)+w, \,\, P^{\star}(X)=(1-u-X)(1-v-X)-w. 
\label{eq:ohno-zagier-proof3.5}
\end{align}
\end{proposition}

\begin{proof}
Set $\Phi_{0}(t;q)=\sum_{j=1}^{n-1}c_{j}t^{j}$ and substitute it into the equation \eqref{eq:ohno-zagier-proof1}. 
Then we find that 
\begin{align}
& 
(1-q)(1-q-u)c_{1}=1, 
\label{eq:ohno-zagier-proof3-initial} \\
& 
(1-q^{j+1})(1-u-q^{j+1})c_{j+1}=P(q^{j})c_{j} \quad (1\le j \le n-2), 
\label{eq:ohno-zagier-proof3-recurrence} \\
& 
P(q^{n-1})c_{n-1}=\Phi(1;q).
\label{eq:ohno-zagier-proof3-end} 
\end{align}
{}From the initial condition \eqref{eq:ohno-zagier-proof3-initial} and 
the recurrence relation \eqref{eq:ohno-zagier-proof3-recurrence}, we get 
\begin{align*}
c_{n-1}=\frac{\prod_{j=1}^{n-2}P(q^{j})}{\prod_{j=1}^{n-1}(1-q^{j})(1-u-q^{j})}.  
\end{align*} 
Hence the consistency between the above formula and \eqref{eq:ohno-zagier-proof3-end} implies 
\eqref{eq:ohno-zagier-proof4}.
The verification of \eqref{eq:ohno-zagier-proof5} is similar. 
\end{proof}

\subsubsection{Specialization at a root of unity}

Now we set $q=\zeta_{n}$. 
By Lemma \ref{lem:ohno-zagier-L-to-z}, we know that 
\begin{align*}
\Phi(u, v, w; 1;\zeta_n) =F_{n}(x, y, z), \quad \Phi^\star(u, v, w; 1;\zeta_n) =F_{n}^\star(x, y, z),
\end{align*}
where $u,v,w$ are given by \eqref{eq:ohno-zagier-proof-transform}.
Therefore Theorem \ref{thm:ohno-zagier} follows from the equalities 
\begin{align*}
\Phi(1; \zeta_{n})=U_{n}(x, y, z) \quad \hbox{and} \quad \Phi^{\star}(1; \zeta_{n})=U_{n}(x, -y, -z)^{-1}.  
\end{align*}

Let us rewrite the right hand side of \eqref{eq:ohno-zagier-proof4} with $q=\zeta_{n}$. 
After the change of variables \eqref{eq:ohno-zagier-proof-transform}, 
the polynomial $P(X)$ defined by \eqref{eq:ohno-zagier-proof3.5} becomes 
\begin{align*}
P(X)=\frac{1+y}{1+x}-\left(1+y+\frac{1-z}{1+x}\right)X+X^{2}. 
\end{align*}
Now we introduce two variables $\alpha$ and $\beta$ such that 
\begin{align*}
\alpha+\beta=(1+x)(1+y)+1-z, \qquad \alpha\beta=(1+x)(1+y). 
\end{align*}
Then it holds that 
\begin{align*}
P(X)=\left(\frac{\alpha}{1+x}-X\right)\left(\frac{\beta}{1+x}-X\right).   
\end{align*}
Using \eqref{eq:ohno-zagier-proof-transform} and the equality 
$\prod_{j=1}^{n-1}(T-\zeta_n^{j})=(T^{n}-1)/(T-1)$, 
we can rewrite the right hand side of \eqref{eq:ohno-zagier-proof4} at $q=\zeta_n$ as 
\begin{align*}
\prod_{j=1}^{n-1}\frac{P(\zeta_n^{j})}{(1-\zeta_n^{j})(1-u-\zeta_n^{j})}=\frac{x}{(1+x)^{n}-1}\,\frac{1}{xy-z}\,
\frac{\alpha^{n}+\beta^{n}-(1+x)^{n}-(1+y)^{n}}{n}.  
\end{align*}
Thus we find that 
\begin{align}
\Phi(1; \zeta_{n})=\frac{x}{(1+x)^{n}-1}\,\frac{1}{xy-z}\,
\frac{\alpha^{n}+\beta^{n}-(1+x)^{n}-(1+y)^{n}}{n}.  
\label{eq:ohno-zagier-pf-formula-Fn}
\end{align}
To calculate the last factor we consider the generating function 
\begin{align*}
\sum_{n=1}^{\infty}\frac{\alpha^{n}+\beta^{n}-(1+x)^{n}-(1+y)^{n}}{n}\,T^{n}.   
\end{align*}
It is equal to 
\begin{align*}
& 
\log{\left(\frac{(1-(1+x)T)(1-(1+y)T)}{(1-\alpha T)(1-\beta T)}\right)}=-\log{
\left(1-\frac{(xy-z)T}{(1-(1+x)T)(1-(1+y)T)}\right)
} \\ 
&=\sum_{n=1}^{\infty}\frac{1}{n}\left(\frac{(xy-z)T}{(1-(1+x)T)(1-(1+y)T)}\right)^{n}. 
\end{align*}
Expand it into a power series of $T$. 
Then the coefficient of $T^{n}$ is given by 
\begin{align*}
\sum_{\substack{a, b \ge 0 \\ a+b\le n-1}}
\frac{1}{n-a-b}\binom{n-a-1}{b}\binom{n-b-1}{a}
(1+x)^{a}(1+y)^{b}(xy-z)^{n-a-b}.  
\end{align*}
Thus we get $\Phi(1; \zeta_{n})=U_{n}(x, y, z)$. 

The proof for the remaining equality $\Phi^{\star}(1; \zeta_{n})=U_{n}(x, -y, -z)^{-1}$ is similar. 
The polynomial $P^{\star}(X)$ is obtained from $P(X)$ by the change of variables $(v, w) \to (-v, -w)$. 
It corresponds to the change $(y, z) \to (-y, -z)$ 
under the transform \eqref{eq:ohno-zagier-proof-transform}.  
Since $\Phi(1; \zeta_{n})=U_{n}(x, y, z)$, 
we get $\Phi^{\star}(1; \zeta_{n})=U_{n}(x, -y, -z)^{-1}$. 
This completes the proof of Theorem \ref{thm:ohno-zagier}.
\qed

We end this subsection by considering a specialization of Theorem \ref{thm:ohno-zagier}.

\begin{corollary}\label{cor:generating-function-depth-one-and-sumformula}
(i) It holds that 
\begin{align}
\sum_{k=1}^{\infty}\bz_{n}(k;\zeta_n)x^{k}=\frac{nx}{1-(1+x)^{n}}+1.  
\label{eq:generating-function-depth-one}
\end{align} 
(ii) Suppose that $k \ge r$ and $n>r>0$, and 
denote by $I(k, r)$ the set of indices of weight $k$ and depth $r$. 
Then we have
\begin{align}
\sum_{\mathbf{k} \in I(k, r)}\bz_{n}(\mathbf{k};\zeta_n)=\sum_{j=1}^{r}\frac{1}{n}\binom{n}{j}
\bz_{n}(k+1-j;\zeta_n).  
\label{eq:sum-formula}
\end{align}
\end{corollary}

\begin{proof}
Setting $z=xy$ we have 
\begin{align*}
F_{n}(x, y, xy)=1+\sum_{r=1}^{n-1}\sum_{k=r}^{\infty}
\left(\sum_{\mathbf{k} \in I(k, r)}\bz_{n}(\mathbf{k};\zeta_n)\right)x^{k-r}y^{r}. 
\end{align*}
We rewrite $U_{n}(x, y, xy)$ as follows. 
Since it holds that   
\begin{align*}
& 
\sum_{\substack{a, b\ge 0 \\ a+b=n-1}}\binom{n-a-1}{b}\binom{n-b-1}{a}(1+x)^{a}(1+y)^{b} \\ 
&=\sum_{\substack{a, b\ge 0 \\ a+b=n-1}}(1+x)^{a}(1+y)^{b} \\ 
&=\frac{(1+x)^{n}-1}{x}+\sum_{r=1}^{n-1}y^{r}x^{-r}\left(
\frac{(1+x)^{n}-1}{x}-\sum_{j=1}^{r}\binom{n}{j}x^{j-1}\right),  
\end{align*}
we have 
\begin{align*}
U_{n}(x, y, xy)=1+\sum_{r=1}^{n-1}y^{r}x^{-r}
\left(1-\frac{x}{(1+x)^{n}-1}\sum_{j=1}^{r}\binom{n}{j}x^{j-1}\right).   
\end{align*}
Comparing the coefficient of $y^{r}$ in the equality $F_{n}(x, y, xy)=U_{n}(x, y, xy)$, 
we see that 
\begin{align}
\sum_{k=r}^{\infty}\left(\sum_{\mathbf{k} \in I(k, r)} \bz_{n}(\mathbf{k};\zeta_n)\right)x^{k}=
1-\frac{x}{(1+x)^{n}-1}\sum_{j=1}^{r}\binom{n}{j}x^{j-1}
\label{eq:sum-formula-no-height}
\end{align}
for $r \ge 1$. 
Setting $r=1$ we obtain \eqref{eq:generating-function-depth-one}.

From \eqref{eq:generating-function-depth-one} we see that 
\begin{align*}
\frac{x}{(1+x)^{n}-1}=\frac{1}{n}\left(1-\sum_{l=1}^{\infty} \bz_{n}(l;\zeta_n)x^{l}\right). 
\end{align*}
Substituting it into \eqref{eq:sum-formula-no-height} we obtain the equality 
\begin{align*}
\sum_{k=r}^{\infty}\left(\sum_{\mathbf{k} \in I(k, r)} \bz_{n}(\mathbf{k};\zeta_n)\right)x^{k}=
1-\left(1-\sum_{l=1}^{\infty} \bz_{n}(l;\zeta_n)x^{l}\right)
\left(1+\sum_{j=1}^{r-1}\frac{1}{n}\binom{n}{j+1}x^{j}\right). 
\end{align*}
It implies the desired equality \eqref{eq:sum-formula}.
\end{proof}
%
%

\section{Applications} \label{sec:application}

In this section we give applications to the numbers $\z(\kk)$ and  $\z^\star(\kk)$, which were defined in \cite[Theorem 1.2]{BTT} by the limits
\begin{align*}
\z(\kk)=\lim_{n\to \infty}\zn_n(\kk;e^{2\pi i/n})\,,\qquad \z^\star(\kk)=\lim_{n\to \infty}\zn^\star_n(\kk;e^{2\pi i/n})\,. 
\end{align*}
 The real part of $\z(\kk)$ is congruent to the symmetric multiple zeta values modulo $\zeta(2)$. From \eqref{eq:degbern} it follows that for all $k\geq 1$ we have 
\begin{align}\label{eq:xidep1}
\xi(k) = -\frac{B_k}{k!}(-2\pi i)^k\,,
\end{align}
where $B_k$ is the $k$-th Bernoulli number with the convention $B_1 = -\frac{1}{2}$. As a consequence of Section \ref{sec:zkkk} we obtain the following.
\begin{corollary}
For all $k,r\geq 1$  we have $\z({\{k\}^r}) \in (-2\pi i)^{kr}\Q $ and in particular
\begin{align*}
\z({\{1\}^r}) &= \frac{(-2\pi i)^r}{r+1}\,, \nonumber \\ 
\z({\{2\}^r}) &= \frac{2^{2r} \pi^{2r}}{(r+1) (2r+1)!}= \frac{2^{2r}}{r+1} \zeta(\{2\}^r)\,, \\
\z({\{3\}^r}) &= \frac{(1+(-1)^r) (-2\pi i)^{3r}}{(r+1)(3r+2)!}  \,.
\nonumber 
\end{align*}
\end{corollary}
\begin{proof} This follows directly from Theorem \ref{thm:zkkk} together with the fact that $ n (1-e^{2\pi i/n})$ goes to $-2\pi i$ as $n \rightarrow \infty$.
\end{proof}

As a consequence of the results in Section \ref{sec:ohnozagier} we obtain the following Ohno-Zagier-type relations for the values  $\z(\kk)$.
\begin{proposition}\label{prop:ohno-zagier-z}
For positive integers $k,r$ and $s$, 
the sum of $\z(\kk)$ (or $\z^{\star}(\kk)$) over $I(k, r, s)$ belongs to $(-2\pi i)^{k}\,\Q$. 
More explicitly we have the following formula for the generating functions. 
Define 
\begin{align*}
& 
\widetilde{F}(x, y, z)=1+\sum_{r=1}^{\infty}\sum_{s=0}^{r}\sum_{k=r+s}^{\infty}
\left((-2\pi i)^{-k}\sum_{\kk \in I(k, r, s)}\z(\kk)\right)x^{k-r-s}y^{r-s}z^{s}, \\ 
& 
\widetilde{F}^{\star}(x, y, z)=1+\sum_{r=1}^{\infty}\sum_{s=0}^{r}\sum_{k=r+s}^{\infty}
\left((-2\pi i)^{-k}\sum_{\kk \in I(k, r, s)}\z^{\star}(\kk)\right)x^{k-r-s}y^{r-s}z^{s}. 
\end{align*} 
Then we have 
\begin{align*}
\widetilde{F}(x, y, z)=\widetilde{U}(x, y, z), \qquad 
\widetilde{F}^{\star}(x, y, z)=\widetilde{U}(x, -y, -z)^{-1},   
\end{align*}
where 
\begin{align*}
\widetilde{U}(x, y, z)=e^{y/2}\frac{x}{\sinh{(x/2)}}
\frac{\cosh{(\frac{1}{2}\sqrt{(x+y)^{2}-4z})}-\cosh{(\frac{1}{2}(x-y))}}{xy-z}. 
\end{align*}
\end{proposition}

\begin{proof}

Let $F_{n}(x, y, z)$ be the generating function defined by \eqref{eq:ohno-zagier-generating-funtcion}. 
The coefficient of $x^{k-r-s}y^{r-s}z^{s}$ in $F_{n}(x/n, y/n, z/n^{2})$ is equal to 
$n^{-k}\sum_{\kk \in I(k, r, s)} \bz_{n}(\kk;e^{2\pi i/n})$. 
{}From the definition of $\z(\kk)$, it holds that 
\begin{align*}
\lim_{n \to \infty}n^{-\wt(\kk)}\bz_{n}(\kk;e^{2\pi i/n})=(-2\pi i)^{-\wt(\kk)}\z(\kk).  
\end{align*}
Hence we see that 
\begin{align}
& 
(-2\pi i)^{-k}\sum_{\kk \in I(k, r, s)}\z(\kk) 
\nonumber \\ 
&=\lim_{n \to \infty}
\oint x^{-(k-r-s)-1}\frac{dx}{2\pi i}
\oint y^{-(r-s)-1}\frac{dy}{2\pi i}
\oint z^{-s-1}\frac{dz}{2\pi i}\, 
U_{n}(\frac{x}{n}, \frac{y}{n}, \frac{z}{n^{2}}), 
\label{eq:ohno-zagier-gamma-pf1}
\end{align}
where $U_{n}$ is defined by \eqref{eq:ohno-zagier-def-U} and 
$\oint$ denotes the integration around the origin. 

To justify the interchange of limit and integration 
we estimate $U_{n}(x/n, y/n, z/n^{2})$. 
There exists a positive constant $c$ such that 
\begin{align*}
0<\frac{e^{t}-(1+t)}{t} \le ct \qquad (0<t\le 1). 
\end{align*} 
We choose a positive constant $\epsilon$ so that $c\,\epsilon<1/2$, 
and the integration contours to be $|x|=\epsilon, |y|=\epsilon$ and $|z|=\epsilon^{2}$. 
We decompose into two parts $U_{n}(x/n, y/n, z/n^{2})=J_{1}J_{2}$, where  
\begin{align*}
& 
J_{1}=\frac{x}{(1+x/n)^{n}-1}, \\ 
& 
J_{2}=\frac{1}{n}\sum_{\substack{a, b \ge 0 \\ a+b\le n-1}}\frac{1}{n-a-b}
\binom{n-a-1}{b}\binom{n-b-1}{a}
\left(1+\frac{x}{n}\right)^{a}
\left(1+\frac{y}{n}\right)^{b}
\left(\frac{xy-z}{n^{2}}\right)^{n-1-a-b}.   
\end{align*}

First we see that, on the circle $|x|=\epsilon$,  
\begin{align*}
\left|J_{1}^{-1}-1\right|\le 
\frac{1}{\epsilon} \left\{ \left(1+\frac{\epsilon}{n}\right)^{n}-(1+\epsilon) \right\}
\le \frac{1}{\epsilon}(e^{\epsilon}-1-\epsilon) \le c\epsilon <\frac{1}{2}.  
\end{align*}
Hence $|J_{1}|<2$. 
Second $J_{2}$ is estimated by using $|x|=|y|=\epsilon$ and $|z|=\epsilon^{2}$ as 
\begin{align*}
\left|J_{2}\right| &\le 
\frac{1}{n}\sum_{\substack{a, b \ge 0 \\ a+b\le n-1}}\frac{1}{n-a-b}
\binom{n-a-1}{b}\binom{n-b-1}{a}
\left(1+\frac{\epsilon}{n}\right)^{a+b}
\left(\frac{2\epsilon^{2}}{n^{2}}\right)^{n-1-a-b} \\ 
&\le 
\frac{1}{n}\sum_{m=0}^{n-1}
\left(1+\frac{\epsilon}{n}\right)^{m}
\left(\frac{2\epsilon^{2}}{n^{2}}\right)^{n-m-1} 
\sum_{\substack{a, b \ge 0 \\ a+b=m}}
\binom{n-a-1}{b}\binom{n-b-1}{a} \\ 
&\le 
e^{\epsilon} 
\sum_{m=0}^{n-1} \frac{1}{n^{n-m}} \binom{2n-m-1}{m} 
\left(\frac{2\epsilon^{2}}{n}\right)^{n-m-1}.  
\end{align*}
Here we used
\begin{align*}
\sum_{\substack{a, b \ge 0 \\ a+b=m}}\binom{n-a-1}{b}\binom{n-b-1}{a}=\binom{2n-m-1}{m} \qquad 
(n>m\ge 0),
\label{eq:binom-prod-simplify} 
\end{align*}
 to obtain the last inequality.  
Since 
\begin{align*}
\frac{1}{n^{n-m}} \binom{2n-m-1}{m}&=
\binom{n-1}{m} \frac{(n-m-1)!}{(2n-2m-1)!}  \prod_{a=1}^{n-m-1}\! \left(2-\frac{m+a}{n}\right) \\ 
&\le \binom{n-1}{m} 
\frac{2^{n-m-1}(n-m-1)!}{(2n-2m-1)!} \le \binom{n-1}{m}, 
\end{align*}
we find that 
\begin{align*}
\left|J_{2}\right| \le  e^{\epsilon} 
\sum_{m=0}^{n-1} \binom{n-1}{m}
\left(\frac{2\epsilon^{2}}{n}\right)^{n-m-1}=
e^{\epsilon} \left(1+\frac{2\epsilon^{2}}{n}\right)^{n-1} \le 
e^{\epsilon+2\epsilon^{2}}. 
\end{align*}
Thus we see that 
\begin{align*}
\left|U_{n}(\frac{x}{n}, \frac{y}{n}, \frac{z}{n^{2}})\right| \le 2 e^{\epsilon+2\epsilon^{2}} 
\end{align*}
on $|x|=|y|=\epsilon, |z|=\epsilon^{2}$. 
Therefore we can interchange the limit and the integration of \eqref{eq:ohno-zagier-gamma-pf1}. 
As a consequence we find that 
\begin{align*}
\widetilde{F}(x, y, z)=\lim_{n \to \infty}U_{n}\left(\frac{x}{n}, \frac{y}{n}, \frac{z}{n^{2}}\right).  
\end{align*}

To calculate the limit in the right hand side, 
we make use of the expression \eqref{eq:ohno-zagier-pf-formula-Fn}. 
It holds that 
\begin{align*}
U_{n}\left(\frac{x}{n}, \frac{y}{n}, \frac{z}{n^{2}}\right)=\frac{1}{xy-z}\frac{x}{(1+x/n)^{n}-1}
\left\{\tilde{\alpha}^{n}+\tilde{\beta}^{n}-\left(1+\frac{x}{n}\right)^{n}-\left(1+\frac{y}{n}\right)^{n}\right\},    
\end{align*}
where $\tilde{\alpha}$ and $\tilde{\beta}$ are determined from 
\begin{align*}
\tilde{\alpha}+\tilde{\beta}=\left(1+\frac{x}{n}\right)\left(1+\frac{y}{n}\right)+1-\frac{z}{n^{2}}, \quad 
\tilde{\alpha}\tilde{\beta}=\left(1+\frac{x}{n}\right)\left(1+\frac{y}{n}\right). 
\end{align*}
Then we see that the asymptotic behavior of $\tilde{\alpha}$ and $\tilde{\beta}$ are given by 
\begin{align*}
1+\frac{1}{2n}\left(x+y\pm \sqrt{(x+y)^{2}-4z}\right)+o(\frac{1}{n}) \qquad (n \to \infty).  
\end{align*}
Hence we find that 
\begin{align*}
& 
\lim_{n \to \infty}U_{n}\left(\frac{x}{n}, \frac{y}{n}, \frac{z}{n^{2}}\right)=
\frac{1}{xy-z}\frac{x}{e^{x}-1}
\left\{2e^{(x+y)/2}\cosh{(\frac{1}{2}\sqrt{(x+y)^{2}-4z})}-e^{x}-e^{y}\right\} \\ 
&=\frac{1}{xy-z}\frac{xe^{y/2}}{\sinh{(x/2)}}
\left\{\cosh{(\frac{1}{2}\sqrt{(x+y)^{2}-4z})}-\cosh{(\frac{1}{2}(x-y))}\right\}. 
\end{align*}
Thus we get the equality $\widetilde{F}(x, y, z)=\widetilde{U}(x, y, z)$. 

By the same argument as before we see that 
\begin{align*}
\widetilde{F}^{\star}(x, y, z)^{-1}=\lim_{n \to \infty}U_{n}\left(\frac{x}{n}, -\frac{y}{n}, -\frac{z}{n^{2}}\right)  \,.
\end{align*}
Hence it also holds that 
$\widetilde{F}^{\star}(x, y, z)=\widetilde{U}(x, -y, -z)^{-1}$.
\end{proof}

We end this note by giving the proof of the sum-formula for the $\xi(\kk)$.
\begin{proof}[Proof of Corollary \ref{cor:xisum}]
Using Stirling's formula, we see that
\begin{align*}
\lim\limits_{n\to \infty}\frac{(1-e^{2\pi i/n})^{j-1}}{n}\binom{n}{j} = \frac{(-2\pi i)^{j-1}}{j!}. 
\end{align*}
Combining this with Corollary \ref{cor:generating-function-depth-one-and-sumformula}, 
we obtain
\begin{align*}
\sum_{\kk \in I(k, r)}\z(\kk)=\sum_{j=1}^{r}\frac{(-2\pi i)^{j-1}}{j!} \z(k+1-j).
\end{align*} 
Using \eqref{eq:xidep1} for the evaluation of $\z(k+1-j)$ we obtain the formula in Corollary \ref{cor:xisum}.
\end{proof}

%
%
%
%
%



\end{document}